\documentclass[12pt]{amsart}
 \usepackage{amscd,amsmath,amsthm,amssymb}
 \usepackage{pstricks,pst-plot,pst-3d}

 \psset{unit=0.7cm,linewidth=0.8pt,arrowsize=2.5pt 4}

\newpsstyle{fatline}{linewidth=1.5pt}
\newpsstyle{fyp}{fillstyle=solid,fillcolor=verylight}
\definecolor{verylight}{gray}{0.97}
\definecolor{light}{gray}{0.93}
\definecolor{medium}{gray}{0.82}
 \def\ab{{\bold a}}
 
 \def\xb{{\bold x}}

 \def\opn#1#2{\def#1{\operatorname{#2}}} 
 \opn\chara{char} \opn\length{\ell} \opn\pd{pd} \opn\rk{rk}
 \opn\projdim{proj\,dim} \opn\injdim{inj\,dim} \opn\rank{rank}
 \opn\depth{depth} \opn\grade{grade} \opn\height{height}
 \opn\embdim{emb\,dim} \opn\codim{codim}
 
 \opn\Tr{Tr} \opn\bigrank{big\,rank}
 \opn\superheight{superheight}\opn\lcm{lcm}
 \opn\trdeg{tr\,deg}
 \opn\reg{reg} \opn\lreg{lreg} \opn\ini{in} \opn\lpd{lpd}
 \opn\size{size} \opn\sdepth{sdepth}
 \opn\link{link}\opn\fdepth{fdepth}\opn\lex{lex}
 \opn\BG{BG} \opn\multideg{multideg}  \opn\LG{LG}
 \opn\div{div} \opn\Div{Div} \opn\cl{cl} \opn\Cl{Cl}
 \opn\Spec{Spec} \opn\Supp{Supp} \opn\supp{supp} \opn\Sing{Sing}
 \opn\Ass{Ass} \opn\Min{Min}\opn\Mon{Mon}
 \opn\Ann{Ann} \opn\Rad{Rad} \opn\Soc{Soc}
 \opn\Im{Im} \opn\Ker{Ker} \opn\Coker{Coker} \opn\Am{Am}
 \opn\Hom{Hom} \opn\Tor{Tor} \opn\Ext{Ext} \opn\End{End}
 \opn\Aut{Aut} \opn\id{id}
 
 \opn\nat{nat}
 \opn\pff{pf}
 \opn\Pf{Pf} \opn\GL{GL} \opn\SL{SL} \opn\mod{mod} \opn\ord{ord}
 \opn\Gin{Gin} \opn\Hilb{Hilb}\opn\sort{sort}
\opn\A{A}
 \opn\aff{aff} \opn
 \con{conv} \opn\relint{relint} \opn\st{st}
 \opn\lk{lk} \opn\cn{cn} \opn\core{core} \opn\vol{vol}
 \opn\link{link} \opn\star{star}\opn\lex{lex}\opn\set{set}
 \opn\gr{gr}
 
 \def\pot#1#2{#1[\kern-0.28ex[#2]\kern-0.28ex]}

 \opn\dirlim{\underrightarrow{\lim}}
 \opn\inivlim{\underleftarrow{\lim}}
 \let\union=\cup

 \def\Implies{\ifmmode\Longrightarrow \else
         \unskip${}\Longrightarrow{}$\ignorespaces\fi}
 \def\implies{\ifmmode\Rightarrow \else
         \unskip${}\Rightarrow{}$\ignorespaces\fi}
 \def\iff{\ifmmode\Longleftrightarrow \else
         \unskip${}\Longleftrightarrow{}$\ignorespaces\fi}

 \let\:=\colon
 \newtheorem{Theorem}{Theorem}[section]
 \newtheorem{Lemma}[Theorem]{Lemma}
 \newtheorem{Corollary}[Theorem]{Corollary}

 \let\epsilon\varepsilon
 \let\kappa=\varkappa
 \def\qed{\ifhmode\textqed\fi
       \ifmmode\ifinner\quad\qedsymbol\else\dispqed\fi\fi}
 \def\textqed{\unskip\nobreak\penalty50
        \hskip2em\hbox{}\nobreak\hfil\qedsymbol
        \parfillskip=0pt \finalhyphendemerits=0}
 \def\dispqed{\rlap{\qquad\qedsymbol}}
 
 \opn\dis{dis}
 \def\pnt{{\raise0.5mm\hbox{\large\bf.}}}
 
 \opn\Lex{Lex}

\begin{document}

 \title {Multigraded Shifts of Matroidal Ideals}

 \author {Shamila Bayati}

\address{Faculty of Mathematics and Computer Science\\
Amirkabir University of Technology (Tehran Polytechnic)\\ Tehran
15914\\ Iran}
\address{School of Mathematics\\ Institute for Research in Fundamental
Sciences (IPM) P. O. Box: 19395-5746\\ Tehran\\ Iran}

\email{shamilabayati@gmail.com}

 \begin{abstract}
In this paper, we show that if $I$ is a matroidal ideal, then the  ideal  generated by   the $i$-th multigraded shifts  is also a matroidal ideal for every $i=0,\ldots,\projdim(I)$.
 \end{abstract}

 \maketitle

 \section*{Introduction}
Matroids, as an interesting subject at the center of combinatorics,  made their way into commutative algebra through several ways including \textit{matroidal ideals}. So it has been of interest to study free resolution of (poly)matroidal ideals and related invariants, see \cite{Ch,CH,JRV}. See also \cite{NPS,T} where matroids and the theory of free resolutions are applied in favor of each other.  If, in particular, we consider the multigraded free resolution which reflects the combinatorial structure, then  more  combinatorial information encoded in the resolution will be found out. 

\medskip

In this paper, our goal is to investigate whether the property of being matroidal is inherited by the ideals generated by multigraded shifts; a question which first came up about the mutligraded shifts of polymatroidal ideals  during a discussion with J\"{u}rgen Herzog and Somayeh Bandari  when we met   in Essen in 2012.

\medskip

 Let  $S=k[x_1,\ldots, x_n]$ be the polynomial ring in the variables $x_1,\ldots, x_n$ over a field $k$. We consider this ring with its natural multigrading where  $\xb^{\ab}=x_1^{a_1}\ldots x_n^{a_n}$ is the unique monomial with multidegree $(a_1,\ldots a_n)$. Throughout, a monomial and its multidegree will be used interchangeably, and $S(\xb^\ab)$ will denote the free $S$-module with one generator of multidegree $\xb^{\ab}$.    A monomial ideal $I\subseteq S$ has a (unique up to isomorphism) minimal multigraded resolution
\[
\mathbf{F}: 0 \rightarrow F_p \rightarrow \ldots \rightarrow F_1 \rightarrow F_0
\]
with
$$F_i=\bigoplus_{{\ab}\in \mathbb{Z}^n}S(\xb^\ab)^{\beta_{i,\ab}}.$$
The set of \textit{$i$-th multigraded shifts} of $I$ is
\[
\{{\xb}^{\ab}| \,\, \beta_{i,\ab}\neq 0 \, \}.
\]
The  ideal generated by the $i$-th multigraded shifts of $I$  will be denoted by $J_i(I)$. We show that $J_i(I)$  can be obtained by taking $i$ times iterated  adjacency ideals starting from  $I$; see Section 2 for relevant definitions. On the other hand, we show that the adjacency ideal of a matroidal ideal is also matroidal. Therefore, we conclude that  if $I$ is a matroidal ideal, then  $J_i(I)$  generated by  its set of $i$-th multigraded shifts  is also a matroidal ideal for every $i=0,\ldots,\projdim(I)$.

 \section{Preliminaries}
In this section, we explain some terminology and facts that we shall use in this  paper.

\medskip

\noindent
\textit{1.1. Multigraded shifts for ideals with linear quotients}

\medskip

Let  $I$ be a monomial ideal in $S=k[x_1,\ldots,x_n]$. We denote its minimal set of monomial generators by $G(I)$. The ideal $I$ is said to have linear quotients if there exists an ordering $m_1,\ldots,m_r$ of the elements of $G(I)$ such that for all $i=1,\ldots,r-1$, the colon ideal $(m_1,\ldots,m_i):(m_{i+1})$ is generated by a subset of $\{x_1,\ldots , x_n\}$. After having fixed such an ordering, the set of each $m_i$ is defined to be 
\[
\set(m_i)=\{k\in [n] \: \,\, x_k\in  (m_1,\ldots,m_i):(m_i) \, \}.
\]

 Let $I$ have linear quotients with order $m_1,\ldots,m_r$ of the elements of $G(I)$. By \cite[Lemma 1.5]{HT} a minimal multigraded free resolution $\mathbf{F}$ of $I$ by mapping cone constructions can be described    as follows:  the $S$-module $F_i$ in  homological degree $i$ of $\mathbf{F}$ is the multigraded free $S$-module whose basis is formed by monomials
$m\xb^{\ab}$  with 
\begin{quote}
\textit{ $m\in G(I)$  and     squarefree monomial  $\xb^{\ab}$  such that $\supp(\xb^{\ab})$ is a subset of $\set(m)$  of cardinality 
 $i$.}
\end{quote}

\medskip

 Recall that we use monomials and their multidegree interchangeably, as well as squarefree monomials and their support. So in the rest of this paper, we will use some notions related to  squarefree monomials for the subsets of $[n]$. In particular, if $B\subseteq [n]$ is the support of a squarefree monomial $m$, we write $\set(B)$ instead of $\set(m)$.

\medskip

\noindent
 \textit{1.2. Matroids}

\medskip
In this subsection, we give the definition of matroids and some basic facts about them. See \cite{O} for a detailed materiel on matroid theory.

\medskip

 We denote  the union of sets $B\cup \{b\}$ by $B+b$, and the difference of sets $B\setminus \{b\}$ by $B-b$.

 A \textit{matroid} $\mathcal{M}$ on a finite set $E$ is a nonempty  collection $\mathcal{B}$ of subsets of $E$, called \textit{bases}, with the following \textit{exchange property}:
\begin{quote}
If $B_1,B_2\in \mathcal{B}$ and $b_1\in B_1\setminus B_2$, then there exists $b_2\in B_2\setminus B_1$ such that $B_1-b_1+b_2\in \mathcal{B}$.
\end{quote}

One may refer to this matroid by $\mathcal{M}=(E,\mathcal{B})$. If $B_2=B_1-b_1+b_2$, it is said that $B_2$ is obtained from $B_1$ by a pivot step: $b_1$ is pivoted out and $b_2$ is pivoted in. It is known that bases of $\mathcal{M}$ possess the following \textit{symmetric exchange property}:
\begin{quote}
If $B_1,B_2\in \mathcal{B}$ and $b_1\in B_1\setminus B_2$, then there exists $b_2\in B_2\setminus B_1$ such that $B_1-b_1+b_2$ and $ B_2-b_2+b_1$ are both in $\mathcal{B}$.
\end{quote}

 It is a known fact that all the bases of a matroid have the same cardinality.
Associated to a matroid $\mathcal{M}=(E,\mathcal{B})$, its \textit{basis graph} $\BG(\mathcal{M})$ is defined as follows: the vertex set of $\BG(\mathcal{M})$ is $\mathcal{B}$,   and two vertices $B_1$ and $B_2$ are adjacent if $|B_1\setminus B_2|=|B_2\setminus B_1|$=1.
If $B_2$   is obtained from    $B_1$ by $b_1$  pivoted out and $b_2$  pivoted in, then an edge label $b_1b_2$  represents  this pivot step diagrammatically as follows:
\begin{center}
\begin{pspicture}(-2,-0.5)(2,1)
\psdots(-1.4,0)(1.2,0)
\psline(-1.4,0)(1.2,0)
\rput(-1.87,0){$B_1$}
\rput(1.87,0){$B_2$}
\rput(0,0.4){$b_1b_2$}
\end{pspicture}
\end{center}
\noindent
 Due to the above mentioned exchange property, matroid basis graphs are connected. A graph  is called a \textit{matroid basis graph} if its vertices could be labeled by bases of some matroid to be its basis graph.

\medskip

Let $G=(V,E)$ be a graph. Recall that the \textit{distance} $d(v_1,v_2)$ between two vertices $v_1,v_2\in V$  is the length of a shortest path between them, if one exists. An \textit{induced subgraph} of $G$ is a graph whose set of vertices is a subset $W$ of $V$ and two vertices of $W$ are adjacent exactly when they are adjacent in $G$. 

\medskip

The following result is clear, for example by the argument applied in the proof of \cite[Lemma 1.4]{Maurer} or directly by the symmetric exchange property for bases of a matroid.

\begin{Lemma}
\label{d=2}
Let $\mathcal{M}=(E,\mathcal{B})$ be a matroid. Suppose that $B_1, B_2\in \mathcal{B}$ are two vertices of the associated basis graph at a distance of two, and   
\[
B_2=B_1-(e_1+e_2)+(f_1+f_2).
\]
 Then $B_1$ and $B_2$  have at least two common  neighbors   
 $C_1, C_2\in \mathcal{B}$  which on the way from $B_1$ to $B_2$,   $e_1$ is pivoted out from one of them and  $e_2$ from the other, and also $f_1$ is pivoted in one of them and  $f_2$ in the other. 
\end{Lemma}


A squarefree monomial ideal $I$ in $S=k[x_1,\ldots,x_n]$  is said to be  a \textit{matroidal ideal} if $$\{\supp(m)| m\in G(I)\}$$ is the set of bases for some matroid. Matroidal ideals are generated in a single degree, and by \cite[Theorem 1.3]{MM} they have linear quotients with respect to the  lexicographical
order of the generators with $x_1>x_2>\cdots>x_n$.

\section{Multigraded shifts of matroidal ideals}
Let  $I$ be a  monomial ideal in $S=k[x_1,\ldots,x_n]$ generated in a single degree. Suppose that $G(I)=\{m_1,\ldots m_r\}$,  and consider the distance between these monomials in the sense of \cite{CH}, that is,
\[
d(m_i,m_j)=\frac{1}{2}\sum_{k=1}^n|\nu_k(m_i)-\nu_k(m_j)|,
\]
where for a monomial $m=x_1^{a_1}\ldots x_n^{a_n}$, one has $\nu_k(m)=a_k$. We associate a graph $G_I$ to $I$  whose  set of vertices is the set of generators  $G(I)$, and two vertices $m_i$ and $m_j$ are adjacent if $d(m_i,m_j)=1$. In particular, $G_I$   can be considered as a matroid basis graph when $I$ is a matroidal ideal.  We define the \textit{adjacency ideal  } of $I$, denoted by $\A(I)$, to be the monomial ideal generated by the  least common multiples of  adjacent vertices in $G_I$, that is,
\[
\A(I)=\langle \lcm(m_i,m_j):\,d(m_i,m_j)=1 \rangle \subseteq k[x_1,\ldots,x_n].
\]
 One should notice that when $I$ is an ideal generated in a  single degree with linear quotients, then  the ideal $J_1(I) $, generated by its set of first multigraded shifts, is exactly  its adjacency ideal. In particular, this implies that if $I$ is a squarefree monomial ideal, then so is $J_1(I)$.

\begin{Lemma}
\label{label.edge}
Let $I$ be a matroidal ideal. Then its adjacency ideal $\A(I)$ is also a matroidal ideal.
\end{Lemma}
\begin{proof}
Let $\mathcal{M}=([n],\mathcal{B})$ be the  matroid corresponding to $I$.  Recall that we use squarefree monomials in $k[x_1,\ldots,x_n]$  and subsets of $[n]$ interchangeably. Consider two arbitrary distinct  elements $B+e, C+f\in \A(I)$ with $B,C\in \mathcal{B}$,  $e\in \set (B)$, and $f\in\set(C)$. Since $e\in \set (B)$, there exists $B'\in \mathcal{B}$ and $e'\in B\setminus B'$ 
 such that
\[
B+e=B'+e'
\]
We check the exchange property for the elements $B+e$ and $C+f$, that is, for each element 
$b\in (B+e)\setminus (C+f)$, we find an element 
\begin{eqnarray}
\label{element}
c\in (C+f)\setminus (B+e) \,\,\,\,\,\,\text{ such that } \,\,\,\,\,\, B+e-b+c\in \A(I).
\end{eqnarray}
For such an element $b$, we may assume that $b\neq e$. Because if $b=e$, we can proceed by the other presentation, namely $B'+e'$.    By this assumption, we have $b\in B\setminus C $. Therefore, by the exchange property for bases of $\mathcal{B}$, the following set is not empty:
\[
T=\{\tilde{c}\in C\setminus B \:\,\,\, B-b+\tilde{c}\in \mathcal {B}\}.
\]
 There exists  two cases:

\textit{Case 1. } 
First suppose that $e\in T$. So $B_1= B-b+e\in \mathcal{B}$.
\noindent
If $B_1=C$, then $f\neq b$ because $B+e$ and $C+f$ are distinct elements. So on one hand, $f\not\in B$. On the other hand, $f\neq e$ because $f\not\in C$. So the choice of  $c=f$ for which $B+e-b+f=C+f\in \A(I)$, has the  required properties mentioned in~(\ref{element}), and we are done.
Otherwise, if $B_1\neq C$,  there exists $b'\in B_1\setminus C$, and consequently by the exchange property, there exists an element $c\in C\setminus B_1$ such that 
\[
B_2=B_1-b'+c\in \mathcal{B}.
\]
 The distinct  vertices $B_1$ and $B_2$ of the basis graph are  adjacent. This implies that $B+e-b+c=B_1\cup B_2\in \A(I)$. Therefore, this element $c$ is an appropriate choice for~(\ref{element}).

\medskip

\textit{Case 2.} Next suppose that there exists an element $c\in T$ such that $c\neq e$. Thus $B_1=B-b+c\in \mathcal{B}$. If $b=e'$, then the vertex $B_1$ becomes adjacent to $B'$  by $e$ pivoted out and $c$ pivoted in; see the induced subgraph in Figure~\ref{0}. Recall that the edge labels show the pivot steps. Since $B'$ and $B_1$ are adjacent, one has 
\[
B+e-b+c=B+e-e'+c=B'\union B_1\in \A(I),
\]
as required in (\ref{element}). 

\medskip

 If $b\neq e'$, then $B'= B_1-(e'+c)+(e+b)$. Therefore, $d(B,B_2)=2$ because $\{e',c\}$ and $\{e,b\}$ are disjoint sets of cardinality two. So by Lemma~\ref{d=2}, there exist a common neighbor $B_2$ of $B_1$ and $B'$ by $e$ pivoted in, namely $B_2=B_1-e'+e$ or $B_2=B_1-c+e$. In any case $B+e-b+c=B_1\union B_2\in \A(I)$, as desired. 
\end{proof} 
\begin{figure}
\begin{center}
\begin{pspicture}(-3,-0.5)(3,4)
\psdots(-2,0)(2,0)(0,3)
\psline(-2,0)(2,0)(0,3)(-2,0)
 \rput(-2.5,-0.15){$B'$}
 \rput(2.4,-0.15){$B_1$}
 \rput(0,3.4){$B$}
\rput{60}(-1.35,1.65){$ee'$}
\rput{-55}(1.45,1.5){$bc{=}e'c$}
\rput(0,-0.4){$ec$}
  \end{pspicture}
\end{center}
\caption{}\label{0}
\end{figure}

\medskip

Considering the  lexicographical order with respect to $x_1>\cdots>x_n$ on the squarefree monomials in $k[x_1,\ldots,x_n]$ , we have an induced total order on subsets of $[n]$. More precisely,  we set $B>_{lex} C$ for subsets $B,C\subseteq [n]$ exactly when  $m_B>_{lex} m_C$ for squarefree monomials $m_B$ and $m_C$  with   $\supp(m_B)= B$ and $\supp(m_C)= C$ .  For simplicity, we denote  $\{i\}>_{lex}\{j\} $ by $i>_{lex}j$ for $i,j\in [n]$ which is the case exactly when  $i<j$ with ordinary order of integers.

\medskip

\begin{Theorem}
\label{main}
Let $I$ be a matroidal ideal. Then the ideal $J_{\ell}(I)$ generated by the set of $\ell$-th multigraded shifts of $I$ is also a matroidal ideal for all $ \ell=0,\ldots ,\projdim(I)$.
\end{Theorem}
\begin{proof}
Let $\mathcal{M}=([n],\mathcal{B})$ be the matroid corresponding to $I$. So
\[
\mathcal{B}=\{\supp(m)\:\,  m\in G(I)\,\}.
\]
Assume that  elements of $G(I)$ are ordered decreasingly in the lexicographical order with respect to $x_1>\cdots>x_n$. With respect to this order of generators we consider  $\set(m)$ for each $m\in G(I)$.  

\medskip

 We prove the assertion by induction on $\ell$. One has $J_0(I)=I$, and so matroidal. The ideal $J_1(I)$ is matroidal as well by Lemma~\ref{label.edge}.  Let $J_{\ell}(I)$ be matroidal for $\ell \leq k$. 
We show that for $k\geq 1$,  $J_{k+1}(I)$ is  the adjacency ideal of $J_k(I)$, and consequently a matroidal ideal by Lemma~\ref{label.edge}.  

\medskip

 First, consider an element  $U=B+e_1+\cdots+e_{k+1} \in J_{k+1}$  with $B\in \mathcal{B}$ and $e_1,\ldots,e_{k+1}\in \set(B)$. Then $U$ is  the union of adjacent vertices $B+e_1+\cdots+e_k$ and $B+e_1+\cdots+e_{k-1}+e_{k+1}$ of the basis graph of  $J_k(I)$. Thus, $U$ is in the adjacency ideal of $J_k(I)$.
\medskip

 Next we show that the adjacency ideal of $J_{k}(I)$ is a subset of $J_{k+1}(I)$.   Let 
\[
U=B+e_1+\cdots+e_k
\]
 and
\[
 V=C+f_1+\cdots+f_k
\]
 be two elements of $J_k(I)$ with $B,C\in \mathcal{B}$, $e_1,\ldots,e_k \in \set(B)$, and $f_1,\ldots,f_k \in \set(C)$. Suppose that $d(U,V )=1$, that is, they are adjacent vertices in the basis graph of $J_{k}(I)$. We show that $U\cup V$, corresponding to the monomial $lcm(\xb_U,\xb_V)$, is indeed  a union of an element of $D\in \mathcal{B}$ and a subset of cardinality $k+1$ of  $\set(D)$, so belongs to $J_{k+1}(I)$. 

\medskip

We have $d(U,V )=1$. Therefore,    $V=C+f_1+\cdots+f_k$   has exactly one element outside of 
 $U=B+e_1+\cdots+e_k$, say $v$. The element $v$  belongs to some $C'\in \mathcal{B}$.   Hence by the symmetric exchange property there exists $b\in B\setminus C'$ such that
\[
B'=B-b+v \in \mathcal{B}.
\]
 If there exists an element $t\in B$ such that $v>_{lex}t$ and $B-t+v\in \mathcal{B}$, then $B-t+v>_{lex}B$, and consequently $v\in \set(B)$. Thus $U\union V=B+e_1+\cdots+e_k+v\in J_{k+1}(I)$, as desired. So we assume that for each $t\in B$, if  $B-t+v\in \mathcal{B}$, then $t>_{lex}v$. In particular, $b>_{lex}v$. This implies that $B>_{lex} B'$ and $b\in \set(B')$. We will show that $e_i\in \set (B')$ for all $i=1,\ldots,k$. So 
\[
U\union V=B'+b+e_1+\cdots+e_k\in J_{k+1},
\]
 and we are done.

For this purpose, fix $i=1,\ldots,k$. Since $e_i\in \set(B)$, there exists $t_i\in B$ such that $e_i>_{lex}t_i$ and $B_i=B-t_i+e_i\in \mathcal{B}$. If $t_i=b$, then $B_i$ and $B'$ are adjacent vertices in the basis graph of $\mathcal{M}$. See Figure~\ref{adjacent}.  More precisely, $B_i=B'-v+e_i$. Notice that $e_i>_{lex} t_i=b >_{lex} v$. Hence  $B_i>_{lex} B'$ and $e_i\in \set(B')$.

\begin{figure}
\begin{center}
\begin{pspicture}(-3,-0.5)(5,3.5)
\psdots(-1.5,0)(1.5,0)(2.5,3)
\psline(-1.5,0)(1.5,0)(2.5,3)(-1.5,0)
 \rput(-2,-0.15){$B'$}
 \rput(1.85,-0.15){$B_i$}
 \rput(2.8,3.2){$B$}
\rput{35}(0,1.65){$vb$}
\rput{80}(2.3,1.5){$e_it_i$}
\rput(0,-0.4){$ve_i$}
  \end{pspicture}
\end{center}
\caption{}\label{adjacent}
\end{figure}

If $t_i\neq b$, then $\{t_i,v\}$ and $\{e_i,b\}$ are disjoint sets of cardinality two. Therefore, $d(B_i,B')=2$ because $B_i=B'-(t_i+v)+(e_i+b)$. By Lemma~\ref{d=2}, there exists a common neighbor  of $B_i$  and $B'$ in the basis graph by $t_i$ pivoted out, namely $B'-t_i+e_i$ or $B'-t_i+b$.  If we have the common neighbor $B'-t_i+e_i\in \mathcal{B}$, then the desired conclusion follows.  Because $e_i>_{lex} t_i$ implies that $B'-t_i+e_i>_{lex} B'$, and we obtain that $e_i\in \set(B')$. So assume  that $B'-t_i+e_i\not\in \mathcal{B}$. Thus there exists the common neighbor $B'-t_i+b\in \mathcal{B}$  by $t_i$ pivoted out. See the induced subgraph in Figure~\ref{nonadj}. Notice that
\[
B'-t_i+b=(B-b+v)-t_i+b=B+v-t_i.
\]
\begin{figure}
\begin{center}
\begin{pspicture}(-3,-2)(4,3.2)
\psdots(-2,0)(-0.3,-1.5)(2,0)(2.5,2.3)(-0.3,2)
\psline(-2,0)(2.5,2.3)(2,0)
\psline[linewidth=1pt,linestyle=dotted](-2,0)(-0.3,2)(2,0)(-0.3,-1.5)(-2,0)
\psline[linewidth=1pt,linestyle=dotted](-0.3,2)(2.5,2.3)(-0.3,-1.5)
\rput{80}(2.6,1.3){$e_it_i$}
 \rput{20}(-0.3,1.28){$vb$}
 \rput(4,2.6){$B{=}B'{-}v{+}b$}
 \rput(2.45,0){$B_i$}
 \rput(-2.4,0){$B'$}
\rput(-0.8,2.4){$B'{-}t_i{+}b$}
\rput(-0.6,-1.8){$B'{-}v{+}e_i$}
  \end{pspicture}
\end{center}
\caption{}\label{nonadj}
\end{figure}
Recall that we assume that whenever $B-t+v\in \mathcal{B}$ for some $t\in B$, then $t>_{lex} v$. Thus $t_i>_{lex} v$. So far we only have common neighbors of $B_i$ and $B'$ by $b$ pivoted in, namely $B=B'-v+b$ and $B'-t_i+b$. Therefore, we must have the common neighbor $B'-v+e_i$ in $\mathcal{B}$ for which $e_i$ pivoted in.
On the other hand, $e_i>_{lex}     t_i     >_{lex}    v$. Hence also in this case, we have $e_i\in \set(B_i)$.
\end{proof}

\medskip

By the argument applied in proof of Theorem~\ref{main}, we have the following result.
\begin{Corollary}
Let $I$ be a matroidal ideal. Then $J_{\ell}(I)$ is obtained by taking $\ell$ times iterated adjacency ideals starting from $I$.
\end{Corollary}

{}

 \end{document}